\documentclass[a4paper,USenglish,hideChapterNumber]{lipics-v2021}
\usepackage{craigstyle}
\usepackage{multicol}
\usepackage{proof}

\newtheorem{question}[theorem]{Open Question}

\newcommand{\lgk}[1]{\mathtt{#1}}
\newcommand{\prp}[1]{\textsf{#1}}

\newcommand{\cpl}{\lgk{CPL}}
\newcommand{\ipc}{\lgk{IPC}}
\newcommand{\SF}{\lgk{S4}}
\newcommand{\Kbas}{\lgk{K}}
\newcommand{\Tmod}{\lgk{T}}
\newcommand{\Dmod}{\lgk{D}}
\newcommand{\Sfiv}{\lgk{S5}}
\newcommand{\Grz}{\lgk{Grz}}
\newcommand{\KF}{\lgk{K4}}
\newcommand{\DF}{\lgk{D4}}
\newcommand{\mtl}{\lgk{MTL}}
\newcommand{\NF}{\lgk{N4}^\bot}
\newcommand{\NFmin}{\lgk{N4}}
\newcommand{\NE}{\lgk{N3}^\bot}
\newcommand{\NEmin}{\lgk{N3}}
\newcommand{\Jhn}{\lgk{Jhn}}
\newcommand{\BL}{\lgk{BL}}
\newcommand{\R}{\lgk{R}}
\newcommand{\Rt}{\lgk{R}^{\mathsf{t}}}
\newcommand{\SemRt}{\lgk{SemR}^{\mathsf{t}}}
\newcommand{\RM}{\lgk{RM}}
\newcommand{\RMt}{\lgk{RM}^{\mathsf{t}}}
\newcommand{\FL}{\lgk{FL}}
\newcommand{\FLe}{\lgk{FL_e}}

\newcommand{\FLw}{\lgk{FL_w}}
\newcommand{\FLec}{\lgk{FL_{ec}}}
\newcommand{\FLcw}{\lgk{FL_{cw}}}
\newcommand{\FLew}{\lgk{FL_{ew}}}
\newcommand{\FLcm}{\lgk{FL_{cm}}}
\newcommand{\FLecw}{\lgk{FL_{ecw}}}
\newcommand{\SemFL}{\lgk{SemFL}}

\newcommand{\SemFLcm}{\lgk{SemFL_{cm}}}
\newcommand{\SemRLcm}{\lgk{SemFL^+_{cm}}}
\newcommand{\SemFLecm}{\lgk{SemFL_{ecm}}}
\newcommand{\SemRLecm}{\lgk{SemFL^+_{ecm}}}

\newcommand{\De}{\mathrm{\Delta}}
\newcommand{\Ga}{\mathrm{\Gamma}}

\newcommand*{\pfa}[1]{\mbox{\footnotesize $#1$}} 	 
\newcommand*{\seq}{{\vphantom{A}\Rightarrow{\vphantom{A}}}}

\newcommand{\wlr}{\textup{(i)}}
\newcommand{\wrr}{\textup{(o)}}

\newcommand{\cnr}{\textup{(c)}}

\newcommand{\exr}{\textup{(e)}}

\newcommand{\mglr}{\textup{(m)}}

\newcommand{\Lbas}{\Ln_{\mathsf{bas}}}
\newcommand{\Lmod}{\Ln_{\mathsf{mod}}}
\newcommand{\Lsc}{\Ln_{\mathsf{sc}}}

\newcommand{\var}[1]{\mathsf{Var}(#1)}

\newcommand{\Ln}{\mathcal{L}}
\newcommand{\Var}{\mathsf{Var}}
\newcommand{\Fm}{\mathsf{Fm}}
\newcommand{\Lk}{\lgk{L}}

\newcommand{\meet}{\wedge}
\newcommand{\join}{\vee}
\renewcommand{\phi}{\varphi}
\newcommand{\under}{\backslash}
\newcommand{\ovr}{/}

\title{Interpolation in Non-Classical Logics}

\author{Wesley Fussner}{Institute of Computer Science of the Czech Academy of Sciences}{fussner@cs.cas.cz}{}{The author's work was supported by the Czech Science Foundation project 25-18306M.}

\authorrunning{W. Fussner}
\titlerunning{Interpolation in Non-Classical Logics}

\begin{document}

\maketitle

\begin{abstract}
	This chapter surveys some of the main results on interpolation in several of the most prominent families of non-classical logics. Special attention is given to the distinction between the two most commonly studied variants of interpolation---namely, Craig interpolation and deductive interpolation. Our discussion focuses primarily on how these properties present in families of logical systems taken as a whole, particularly those comprising all axiomatic extensions of any of several notable non-classical logics. We consider a range of important examples: superintuitionistic and modal logics, fuzzy logics, paraconsistent logics, relevant logics, and substructural logics.
\end{abstract}

\tableofcontents


\section{Introduction}

In this chapter, we survey the landscape of interpolation in non-classical logics. We have already seen one case study from the non-classical world---namely the modal logic $\Kbas$---in \refchapter{chapter:modal}. Here we will see many other logical systems in the vicinity of $\Kbas$, as well as a plethora of other, more distant logics.

One theme that becomes apparent throughout our discussion is that non-classical logics present many fine distinctions that do not appear for classical propositional logic $\cpl$, or even for many well-behaved non-classical systems such as $\Kbas$. The most salient of these is that the relationship between interpolation for the consequence relation of a logic and interpolation for the primitive implication connective of that logic (if any) often breaks down, leading to distinct interpolation properties that transparently coincide for $\cpl$. Our focus will be on the two most prominent variants of interpolation, here called the Craig interpolation property and the deductive interpolation property, but we will also consider a number of other interpolation properties that have proven important in particular logical environments. We will also discuss links between several interpolation properties, most often given by different kinds of deduction theorems.

This chapter covers a lot of territory, ranging over an incredibly diverse collection of logical systems that have arisen from many different motivations. To make sense of this diversity, we will frame our discussion in the language of (Tarskian) consequence relations, and we will begin by recalling some of the basics of these as well as several related notions. Due the scope of the topic, many important aspects of interpolation in non-classical logics will be omitted for lack of space. Our discussion will not include several prominent variants of interpolation, such as uniform interpolation (see \refchapter{chapter:uniform}) and Lyndon interpolation (see \refchapter{chapter:prooftheory}), and will for the most part avoid discussion of the connection between interpolation and Beth definability, leaving coverage of these topics to other chapters of this volume.

We will also omit discussion of the role the underlying language plays in shaping interpolation. The importance of this topic is evident already in the case of variants of $\cpl$ itself: If $\cpl$ is formulated with constants for truth and falsity, then $\cpl$ enjoys the Craig interpolation theorem in the following form:
\begin{equation*}
  \parbox{\dimexpr\linewidth-6em}{%
    \strut
     Whenever $\alpha,\beta\in\Fm_\cpl$ and $\vdash_\cpl \alpha\to\beta$, then there exists $\delta\in\Fm_\cpl$ such that $	\vdash_\cpl \alpha\to\delta$, $\vdash_\cpl \delta\to\beta$, and $\var{\delta}\subseteq\var{\alpha}\cap\var{\beta}$.
    \strut
  }
\end{equation*}
However, the truth constants themselves may supply interpolants for some theorems, like $(p\meet\neg p)\to (q\join\neg q)$. When truth constants are missing from the language, Craig interpolation therefore must instead by formulated to exclude these kinds of examples:
\begin{equation*}
  \parbox{\dimexpr\linewidth-6em}{%
    \strut
     Whenever $\alpha,\beta\in\Fm_\cpl$, $\not\vdash_\cpl\neg\alpha$, $\not\vdash_\cpl\beta$, and $\vdash_\cpl \alpha\to\beta$, then there exists $\delta\in\Fm_\cpl$ such that $\vdash_\cpl \alpha\to\delta$, $\vdash_\cpl \delta\to\beta$, and $\var{\delta}\subseteq\var{\alpha}\cap\var{\beta}$.
    \strut
  }
\end{equation*}
The variants of $\cpl$ with and without truth constants are best understood as \emph{different logics} when it comes to interpolation, despite being notational variations of the same system. The lesson here is that precise formulations of the systems in question---and of interpolation itself---can have quite significant ramifications. Indeed, a logic that is lacking interpolation (in some specified form) can often be made more expressive with a mild expansion of the underlying language in a way that restores interpolation. This phenomenon is poorly understood in the general setting of non-classical logics (compare, e.g.,  \cite[Theorem 14]{HandbookModal} and \cite{Montagna2006}), and the possible variations in language are virtually endless. We will thus mostly exclude this important topic from our discussion here.

Most of this chapter consists of a tour of interpolation in particular families of logical systems, usually organized by recalling a few of the most naturally defined logics in those families and then discussing what is known about interpolation in the axiomatic extensions of these. Our tour of these systems touches only on the highlights, focusing on a few of the most salient results and making no effort to give a comprehensive treatment. Many of the theorems we summarize are major results whose proofs are well beyond the scope of this chapter. In most cases, we cannot do justice to even a sketch of these proofs in this lone chapter, so proofs will be entirely omitted. However, we typically indicate the overall proof methodology pursued in a given study. There are, for the most part, three main lines of argumentation: Proof-theoretic investigations anchored in Maehara's method (see \refchapter{chapter:prooftheory} for a full discussion), algebraic methods focused on the link between interpolation and various kinds of amalgamation (discussed extensively in \refchapter{chapter:algebra}), and model-theoretic studies that exploit quantifier-elimination theorems. We give no more than allusions to the techniques at play here, leaving the discussion of methods to other chapters. However, we provide an extensive bibliography in order to point the reader to resources where the main technical work is carried out.

Our overall organizational motif---that is, considering prominent logics in turn and, for each of these, indicating what is known about interpolation in axiomatic extensions---was selected for four reasons. Firstly, and perhaps most obviously, this scheme provides some principled means of systematizing the huge amount of information available in the literature, and in particular giving the reader some idea of what classifications are known. Secondly, the chosen organizational scheme gives a coarse measure of \emph{how robust} interpolation is within certain logical environments: As we will see, certain logics (e.g., intuitionistic logic, Section~\ref{subsec:intuitionistic}) have very few axiomatic extensions with interpolation, whereas other closely related logics (e.g., $\SemRLcm$, Section~\ref{subsec:notexchange}) have many. The third reason we adopt our chosen organizational scheme is also comparative: For some logics, we have an exhaustive classification of axiomatic extensions with various kinds of interpolation (again, intuitionistic logic, Section~\ref{subsec:intuitionistic}), while in other, seemingly analogous cases our techniques fail us (e.g., Johansson's minimal logic, Section~\ref{subsec:minimal}). Side-by-side comparisons of how interpolation operates in the axiomatic extensions of given logics gives some indication of what our methods are capable of, as well as points to many open questions. We explicitly identify several such open questions along the way, hoping to inspire readers to further advance the lines of research articulated here. Fourthly and finally, the logics enjoying interpolation are especially \emph{natural} among their cousins in a given family, as determined by a range of different philosophical motivations; see \refchapter{chapter:philosophy}. The organizational scheme adopted here is thus, we hope, a useful device for practitioners arriving at interpolation via these motivations, though we don't discuss them further here.

\section{Consequence, Interpolation, and Deduction Theorems}\label{sec:basic}

Our discussion of interpolation in non-classical logics will be carried out within the context of (Tarskian) consequence relations, which provide a framework allowing us to compare a range of differing logical systems. After discussing the basics of consequence relations, we will introduce a few of the most relevant interpolation notions for this chapter, and subsequently discuss deduction theorems and how they relate the variants of interpolation we will consider. For background on consequence relations and the other preliminary material discussed in this section, please consult the monographs \cite{CitMur,Font2016}.

\subsection{Consequence Relations}

As is most common in the literature, we will always consider propositional logics that are defined over a denumerable set $\Var$ of propositional variables and a language $\Ln$ consisting of finitely many logical connectives, each of finite (possibly 0) arity. We denote by $\Fm_\Ln$ the collection of all formulas built from $\Var$ using the connectives from $\Ln$. Whenever $\alpha\in\Fm_\Ln$, we denote the set of variables appearing in $\alpha$ by $\Var(\alpha)$. If $\Gamma$ is a set of formulas, we also write $\Var(\Gamma) = \bigcup\{\Var(\alpha) \mid \alpha\in\Gamma\}$.

A \emph{consequence relation} over $\Ln$ is a relation $\vdash$ from $\mathcal{P}(\mathcal{\Fm_\Ln})$ to $\Fm_\Ln$ satisfying the following two conditions, where $\Gamma\cup \Pi\cup \{ \alpha \}\subseteq\Fm_\Ln$:
\begin{itemize}
\item Reflexivity: If $\alpha\in\Gamma$, then $\Gamma\vdash\alpha$.
\item Transitivity: If $\Gamma\vdash \alpha$ and $\Pi\vdash\beta$ for every $\beta\in\Gamma$, then $\Pi\vdash\alpha$.
\end{itemize}
It is immediate that every consequence relation also satisfies monotonicity: If $\Gamma\vdash \alpha$ and $\Gamma\subseteq\Pi$, then $\Pi\vdash\alpha$.

The consequence relations that we consider in this chapter satisfy two further properties:
\begin{itemize}
\item Structurality: If $\Gamma\vdash \alpha$, then $\sigma[\Gamma]\vdash\sigma(\alpha)$ for every substitution $\sigma$.
\item Finitarity: If $\Gamma\vdash\alpha$, then there exists a finite subset $\Gamma'\subseteq\Gamma$ such that $\Gamma'\vdash\alpha$.
\end{itemize}
If $\Gamma,\Pi\subseteq\Fm_\mathcal{L}$ and $\vdash$ is a consequence relation over $\mathcal{L}$, then we write $\Gamma\vdash\Pi$ if $\Gamma\vdash\alpha$ for every $\alpha\in\Pi$.

Aside from our few references to first-order logic in this chapter, we will take \emph{logics} (over the language $\Ln$) to be pairs $\Lk=(\Ln,\vdash_\Lk)$, where $\Ln$ is some finitary language and $\vdash_\Lk$ is a finitary, structural consequence relation over $\Ln$. If $\Lk=(\Ln,\vdash_\Lk)$ is a logic, a formula $\alpha\in\Fm_{\Ln}$ is called a
 \emph{theorem} of $\Lk$ provided that $\emptyset\vdash_\Lk\alpha$, which we usually abbreviate by $\vdash_\Lk\alpha$. By an \emph{extension} of a logic $\Lk = (\Ln,\vdash_\Lk)$ we mean any logic of the form $(\Ln,\vdash'_\Lk)$, where $\vdash'_\Lk$ contains $\vdash_\Lk$. We reserve the term \emph{expansion} for enlargements of the language and will not use `extension' for this. If $\Lk = (\Ln,\vdash_\Lk)$ is a logic, a \emph{fragment} of $\Lk$ is a logic $(\Ln',\vdash_\Lk)$ where $\Ln'\subseteq\Ln$.

There is a long tradition of studying a given logic $\Lk$ by looking at its extensions, particularly those that arise by adding axioms (i.e., as opposed to arbitrary rules). If $\Lk = (\Ln,\vdash_\Lk)$ is a logic, then a logic $\Lk' = (\Ln,\vdash_{\Lk'})$ is called an \emph{axiomatic extension} of $\Lk$ if there is a subset $\Pi\subseteq\Fm_\Ln$ that is closed under substitutions such that, for any $\Gamma\cup\{\alpha\}\subseteq\Fm_\Ln$,
\[
\Gamma\vdash_{\Lk'}\alpha \iff \Gamma,\Pi\vdash_\Lk\alpha.
\]
The only axiomatic extensions of classical propositional logic $\cpl$ are $\cpl$ itself and the trivial logic consisting of all formulas in the language. However, for non-classical logics, there are often numerous proper axiomatic extensions. These always form a lattice when ordered by inclusion of the associated consequence relations. The structure of the lattice of axiomatic extensions of a given logic casts considerable light on the logic itself. In much of this chapter, we will examine, for many prominent non-classical logics $\Lk$, which of the axiomatic extensions of $\Lk$ have various interpolation properties. This sort of inquiry has been carried out in a wide range of logical contexts, and illuminates how interpolation operates in such contexts. Of particular benefit of the wide view taken in this chapter, this mode of investigation also allows for the comparison of interpolation phenomena across varying contexts and thereby gives us one perspective on how interpolation operates on a `big picture' level.

\subsection{Variants of Interpolation}

There are a huge number of interpolation notions that have been considered in the non-classical logic literature, many of which are specialized to particular contexts. In this chapter, we will primarily consider two such notions that make sense for most well-studied logical systems. The first of them, and probably the most familiar, is the \emph{Craig interpolation property} (or \prp{CIP}). It can be formulated for any logic $\Lk = (\Ln,\vdash_\Lk)$ that has an implication-like connective $\to$ in its language, and may be stated as follows:
\begin{equation}
  \tag{\prp{CIP}}\label{eq:CIP}
  \parbox{\dimexpr\linewidth-6em}{%
    \strut
     Whenever $\alpha,\beta\in\Fm_\Ln$ and $\vdash_\Lk \alpha\to\beta$, then there exists $\delta\in\Fm_\Ln$, called a \emph{Craig interpolant}, such that $	\vdash_\Lk \alpha\to\delta$, $\vdash_\Lk \delta\to\beta$, and $\var{\delta}\subseteq\var{\alpha}\cap\var{\beta}$.
    \strut
  }
\end{equation}
Of course, here we are contemplating that there is a particular implication-like connective $\to$ that is primitive in the language, and that there is only one such implication. This need not be the case: Some logics have multiple primitive implication connectives, and some have definable implications that we may want to consider. Indeed, one might speak of the \prp{CIP} for some two-variable term $I(p,q)$ that is definable in the language of $\Ln$ but does not appear as a primitive connective in $\Ln$. We do not dwell on this point, but note that many logics have a form of the \prp{CIP} for some defined implication. For instance, there are continuum-many axiomatic extensions of linear logic that have a version of the \prp{CIP} with respect to the non-standard implication $I(p,q) = \;!p\to!q$; see \cite[Proposition~3.16]{FSLinear}. For most of the logics we will consider, it will be clear from context which implication we are referring to when we discuss the \prp{CIP}.

The preceding comments underscore some of the motivations for having an implication-free formulation of interpolation. The most prominent such formulation is given by the \emph{deductive interpolation property} (or \prp{DIP}), which for a logic $\Lk = (\Ln,\vdash_\Lk)$ provides that:
\begin{equation}
  \tag{\prp{DIP}}\label{eq:DIP}
  \parbox{\dimexpr\linewidth-6em}{%
    \strut
Whenever $\alpha,\beta\in\Fm_\Ln$ and $\alpha\vdash_\Lk\beta$, then there exists $\delta\in\Fm_\Ln$, called a \emph{deductive interpolant}, such that $\alpha\vdash_\Lk\delta$, $\delta\vdash_\Lk\beta$, and $\var{\delta}\subseteq\var{\alpha}\cap\var{\beta}$.
    \strut
    }
  \end{equation}
Deductive interpolation makes sense for any logic, even those that have no implication-like connective at all, and is not otherwise bound up with the language of the particular logical system in question. It is thus a good notion for comparing interpolation phenomena across rather different logical contexts. Note that the \prp{DIP} has sometimes been referred to by other names in the literature: Maksimova sometimes calls it the interpolation principle of deducibility (see \cite{Mak91}), or the \prp{IPD} for short, and Madar\'{a}sz calls it the turnstile Craig interpolation property (see \cite{Mad98}). Sometimes causing confusion, the \prp{DIP} is occasionally simply called the Craig interpolation property.

The relationship between the \prp{CIP} (for some given implication) and the \prp{DIP} is quite complicated. A given logic may have either one, both, or neither of these properties. We will return to the relationship between these two notions shortly in Section~\ref{sec:deduction}.

Several stronger variants of the \prp{CIP} and the \prp{DIP} are commonly studied in the literature. These mostly fall into two categories. First, sometimes additional constraints are placed on the variables appearing in an interpolant, as in Lyndon interpolation (see \refchapter{chapter:prooftheory}). Second, sometimes additional constraints are placed on the interpolant itself, such as demanding that it is in some sense `minimal' or `maximal', as in uniform interpolation (see \refchapter{chapter:uniform}). Other constraints can be placed on the language appearing in the interpolant, such as demanding that certain connectives appear in common; see, e.g., \cite{vBenthem97}. For the most part, we will leave these topics to subsequent chapters, focusing on the \prp{CIP} and \prp{DIP} in their most basic formulation. However, we will touch on a few other versions of interpolation that are important for certain logics, two of which we mention now.

First, the following stronger kind of interpolation is sometimes considered in the literature. Here we will call it the \emph{Maehara interpolation property} (or \prp{MIP}), but elsewhere it is sometimes called the strong deductive interpolation property or the \prp{SDIP}. We say that a logic $\Lk = (\Ln,\vdash_\Lk)$ has the Maehara interpolation property provided that:
\begin{equation}
  \tag{\prp{MIP}}\label{eq:MIP}
  \parbox{\dimexpr\linewidth-6em}{%
    \strut
Whenever $\Gamma\cup\Pi\cup\{\alpha\}\subseteq\Fm_\Ln$ and $\Gamma,\Pi\vdash_\Lk\alpha$, then there exists $\delta\in\Fm_\Ln$ such that $\Gamma\vdash_\Lk\delta$ and $\delta,\Pi\vdash_\Lk\alpha$, and also $\var{\delta}\subseteq\var{\Gamma}\cap\var{\Pi\cup\{\alpha\}}$.
    \strut
    }
  \end{equation}
Obviously, the \prp{MIP} implies the \prp{DIP} for any logic.

Second, in certain logics for which the \prp{DIP} fails, it is sometimes useful to consider a slightly weaker variant of the \prp{MIP}. Assume that $\Lk = (\Ln,\vdash_\Lk)$ is a logic with a falsity constant $\bot$ appearing in its language. We say that $\Lk$ has the \emph{weak interpolation property} (or \prp{WIP}) provided that:
\begin{equation}
  \tag{\prp{WIP}}\label{eq:WIP}
  \parbox{\dimexpr\linewidth-6em}{%
    \strut
If $\alpha,\beta\in\Fm_{\Ln}$ and $\alpha,\beta\vdash_\lgk{L}\bot$, then there is a formula $\delta$ such that $\alpha\vdash_\lgk{L}\delta$, $\delta,\beta\vdash_\lgk{L}\bot$, and $\var{\delta}\subseteq\var{\alpha}\cap\var{\beta}$.
    \strut
    }
  \end{equation}

\subsection{Deduction Theorems and Relations Among Interpolation Variants}\label{sec:deduction}

Deduction theorems link deducibility to theoremhood, and thus unsurprisingly play a significant role in connecting the \prp{CIP} and the \prp{DIP}. Like interpolation, deduction theorems come in a range of different forms that are often specialized to particular logical contexts. The most familiar of these are \emph{classical deduction theorems}, which mirror the deduction theorem found in $\cpl$. We say that a logic $\Lk = (\Ln,\vdash_\Lk)$ has a classical deduction theorem provided that, for any $\Gamma\cup\{\alpha,\beta\}\subseteq\Fm_\Ln$,
\[
\Gamma,\alpha\vdash_\Lk\beta \iff \Gamma\vdash_\Lk\alpha\to\beta.
\]
One may readily show that if $\Lk$ has a classical deduction theorem, then $\Lk$ has the \prp{CIP} if and only if $\Lk$ has the \prp{DIP}. This is true, for example, for $\cpl$ as well as propositional intuitionistic logic (see Section~\ref{subsec:intuitionistic} below) and for many other logical systems.

Moving to a slightly higher level of abstraction, we say that $\Lk = (\Ln,\vdash_\Lk)$ has an \emph{explicit deduction theorem} if there is a formula $I(p,q)$ in $\Ln$ in at most two variables $p,q$ such that for any $\Gamma\cup\{\alpha,\beta\}\subseteq\Fm_\Ln$,
\[
\Gamma,\alpha\vdash_\Lk\beta \iff \Gamma\vdash_\Lk I(\alpha,\beta).
\]
The classical deduction theorem takes $I(p,q)$ to be $p\to q$. On the other hand, some logics have an explicit deduction theorem but not a classical one. For instance, the global consequence relation of the modal logic $\SF$ (see Section~\ref{subsec:modal}) satisfies the following explicit deduction theorem:
\[
\Gamma,\alpha\vdash_\Lk\beta \iff \Gamma\vdash_\Lk \Box\alpha\to\beta.
\]
Actually, this holds for any axiomatic extension of $\SF$ and this shows that having an explicit deduction theorem is insufficient to guarantee that the \prp{DIP} implies the \prp{CIP} for some given logic (as we will see in Section~\ref{subsec:modal}, there are extensions of $\SF$ with the \prp{DIP} but not the \prp{CIP}).

Moving up another layer of abstraction, we say that a logic $\Lk = (\Ln,\vdash_\Lk)$ has a \emph{local deduction theorem} if there is a set $\Lambda$ of formulas in at most two variables such that for any $\Gamma\cup\{\alpha,\beta\}\subseteq\Fm_\Ln$,
\[
\Gamma,\alpha\vdash_\Lk\beta \iff \Gamma\vdash_\Lk I(\alpha,\beta)\text{ for some }I\in\Lambda.\textsuperscript{1}\setcounter{footnote}{1}
\]
The global consequence relation of the normal modal logic $\KF$ (see \refchapter{chapter:modal} and Section~\ref{subsec:modal} below) has a local deduction theorem but not an explicit one. Many of the logics we consider have a local deduction theorem of a particularly nice form: Each of the formulas $I(p,q)\in\Lambda$ has the form $t(p)\to q$ for some formula $t$ in one variable. We will call these \emph{implicative local deduction theorems}, though this is not a standard term in the literature.

\footnotetext{Actually, in abstract algebraic logic, often entire sets of formulas are taken to stand in place of individual formulas acting as an implication-like connective. The practicality of this is uncertain; see the discussion on uniterm vs. multiterm deduction detachment theorems at \cite[p.~166]{Font2016}.}

The relevance of implicative local deduction theorems to our discussion is summarized in the following two results. They are both folklore but mostly only proven in special cases (see, e.g., \cite{KO10}), so we include proofs for completeness.

\begin{proposition}
Let $\Lk = (\Ln,\vdash_\Lk)$ be a logic with an implication connective $\to$ in its language, and assume that:
\begin{enumerate}
\item $\Lk$ satisfies modus ponens, i.e., $\alpha,\alpha\to\beta\vdash_\Lk\beta$ for any $\alpha,\beta\in\Fm_\Ln$.
\item $\Lk$ has an implicative local deduction theorem given by a set of formulas $\Lambda$.
\end{enumerate}
Then if $\Lk$ has the \prp{CIP}, then $\Lk$ has the \prp{DIP}.
\end{proposition}

\begin{proof}
Suppose that $\alpha\vdash_\Lk\beta$. Then there exists $t(p)\to q\in\Lambda$ such that $t(\alpha)\to\beta$ is a theorem of $\Lk$, so by the \prp{CIP} there exists $\delta\in\Fm_\Ln$ such that $\vdash_\Lk t(\alpha)\to\delta$, $\vdash_\Lk \delta\to \beta$, and $\var{\delta}\subseteq\var{\alpha}\cap\var{\beta}$. By the deduction theorem, $\vdash_\Lk t(\alpha)\to\delta$ implies $\alpha\vdash_\Lk\delta$. On the other hand, from modus ponens, $\vdash_\Lk\delta\to\beta$ implies $\delta\vdash_\Lk\beta$.
\end{proof}
Most well-studied logics that have local deduction theorems actually have implicative local deduction theorems in the above sense, and we will sometimes just say that these logics have local deduction theorems in our discussion. With the exception of logic of paradox, discussed in Section~\ref{subsec:paraconsistent}, all of the logics we consider in this chapter validate modus ponens.

In some exceptional contexts, the \prp{DIP} implies the \prp{CIP} as well. This is the case, for example, in the presence of a classical deduction theorem, but also in certain cases when an explicit deduction theorem is absent (such as the case of the relevant logic $\lgk{R}$-mingle, as formulated with truth constants; see \cite{MarMet2012}). The precise conditions under which this occurs remain a mystery.

\begin{question}
Under what circumstances does the \prp{DIP} imply the \prp{CIP}? 
\end{question}

The next proposition concerns logics that additionally have a (not necessarily primitive) conjunction-like connective. Given a logic $\Lk = (\Ln,\vdash_\Lk)$, we say that a formula $C(p,q)$ in at most two variables is a \emph{conjunction} for $\Lk$ if for every set of formulas $\Gamma\cup\{\alpha,\beta,\gamma\}$,
\[
\Gamma,\alpha,\beta\vdash_\Lk \gamma \iff \Gamma, C(\alpha,\beta)\vdash_\Lk \gamma.
\]
Of course, in many familiar logics, $C$ may be taken to be the primitive conjunction connective in the language, but the above formulation makes sense even for logics without such a primitive conjunction connective. When $\Lk$ has a conjunction, we say that $\Lk$ is a \emph{conjunctive} logic.

\begin{proposition}
Let $\Lk = (\Ln,\vdash_\Lk)$ be a logic with an implication connective $\to$ in its language, and assume that:
\begin{enumerate}
\item $\Lk$ satisfies modus ponens, i.e., $\alpha,\alpha\to\beta\vdash_\Lk\beta$ for any $\alpha,\beta\in\Fm_\Ln$.
\item $\Lk$ has an implicative local deduction theorem given by a set of formulas $\Lambda$.
\item $\Lk$ is conjunctive.
\end{enumerate}
Then $\Lk$ has the \prp{DIP} if and only if $\Lk$ has the \prp{MIP}.
\end{proposition}

\begin{proof}
It is clear that the \prp{MIP} implies the \prp{DIP}. For the converse, suppose that $\Gamma\cup\Pi\cup\{\alpha\}\subseteq\Fm_\Ln$ and that $\Gamma,\Pi\vdash_\Lk \alpha$. By finitarity, there exist finite subsets $\Gamma'\subseteq\Gamma$ and $\Pi'\subseteq\Pi$ such that $\Gamma',\Pi'\vdash_\Lk \alpha$. Since $\Lk$ is conjunctive, by iterating $C$ there are formulas $C(\Gamma'),C(\Pi')$ so that $C(\Gamma'),C(\Pi')\vdash_\Lk\alpha$. By the implicative local deduction theorem, there exists a formula $t(p)$ in at most one free variable such that $C(\Gamma')\vdash_\Lk t(C(\Pi'))\to\alpha$, and by the \prp{DIP} there exists a formula $\delta$ such that
\[
\var{\delta}\subseteq \var{C(\Gamma')}\cap\var{t(C(\Pi'))\to\alpha} = \var{\Gamma'}\cup\var{\Pi'\cup\{\alpha\}} \subseteq \var{\Gamma}\cap\var{\Pi\cup\{\alpha\}}
\]
and $C(\Gamma')\vdash_\Lk\delta$ and $\delta\vdash t(C(\Pi'))\to\alpha$ hold. Applying the deduction theorem and conjunctivity again gives $\Gamma'\vdash_\Lk\delta$ and $\delta,\Pi'\vdash_\Lk\alpha$, so $\Gamma\vdash_\Lk\delta$ and $\delta,\Pi\vdash_\Lk\alpha$ hold by monotonicity.
\end{proof}

\section{Fundamental Examples}

The remainder of this chapter is devoted to a survey of what is known about interpolation in several prominent families of non-classical logics. The account offered here is neither comprehensive nor representative, but rather intended as a pointer to a few of the landmarks in the vast literature on interpolation in non-classical logics. The proofs of the many of the results conveyed here are extremely involved, and are omitted. For some of the high-level ideas used in these proofs, please consult \refchapter{chapter:modal}, \refchapter{chapter:prooftheory}, \refchapter{chapter:algebra}, and \refchapter{chapter:uniform}.

\subsection{Superintuitionistic Logics}\label{subsec:intuitionistic}

We begin our discussion of fundamental examples with propositional intuitionistic logic $\ipc$ and its axiomatic extensions, which are often called \emph{superintuitionistic logics} or \emph{intermediate logics}. The logic $\ipc$ is usually formulated in the basic  language $\Lbas = \{\meet,\join,\to,\bot,\top\}$, where $\neg\alpha$ and $\alpha\leftrightarrow\beta$ are introduced as abbreviations for $\alpha\to \bot$ and $(\alpha\to\beta)\meet (\beta\to\alpha)$, respectively.

The consequence relation of $\ipc$ can be specified in a number ways. Syntactically, it may be presented by a Hilbert-style calculus (see \cite[Section 2.6]{cha:mod}) or an analytic sequent calculus such as Gentzen's $\mathsf{LJ}$ (see \cite{Gentzen1934} and \refchapter{chapter:prooftheory}). Semantically, $\vdash_\ipc$ can be captured in terms of the intuitionistic Kripke frames or Heyting algebras, which respectively provide relational and equivalent algebraic semantics for $\ipc$. The reader may consult \cite{cha:mod} for a detailed account.

Crucially, each axiomatic extension $\lgk{L}$ of $\ipc$ enjoys a classical deduction theorem
\[
\Gamma,\alpha\vdash_\lgk{L}\beta \iff \Gamma\vdash_\lgk{L}\alpha\to\beta,
\]
so the \prp{CIP} for $\to$ and the \prp{DIP} coincide for all of these logics, and also coincide with the Maehara interpolation property \prp{MIP}.

The first result on interpolation for intuitionistic logic is due to Schütte, who established (the predicate version of) the Craig interpolation property for intuitionistic predicate logic in \cite{Schutte1962}. Schütte's methods were proof theoretic, and were subsequently sharpened by Nagashima \cite{Nagashima1966}, who also used proof-theoretic tools. Gabbay later provided a semantic proof of the Craig interpolation property for intuitionistic predicate logic using Kripke models; see \cite{Gabbay1971}.

Returning to the propositional setting, Maksimova studied interpolation for $\ipc$ and its extensions in the 1970s, leading to an exhaustive description of axiomatic extensions of $\ipc$ with the \prp{CIP} and thus, equivalently, the \prp{DIP} and the \prp{MIP}. This is summarized in the following theorem. For context, note that by Jankov's theorem, there are continuum-many axiomatic extensions of $\ipc$.

\begin{theorem}[\cite{Mak77}]\label{thm:IPC}
There are exactly eight superintuitionistic logics with the \prp{CIP} or, equivalently, with the \prp{DIP} or \prp{MIP}. These eight logics are:
\begin{enumerate}
\item $\ipc$ itself.
\item Classical Propositional Logic $\cpl$, axiomatized relative to $\ipc$ by $\alpha\join\neg\alpha$.
\item G\"odel-Dummett Logic $\lgk{LC}$, axiomatized relative to $\ipc$ by $(\alpha\to\beta)\join (\beta\to\alpha)$.
\item Jankov's Logic $\lgk{KC}$, axiomatized relative to $\ipc$ by $\neg\alpha\join\neg\neg\alpha$.
\item The logic $\lgk{LP}_2$, axiomatized relative to $\ipc$ by $\alpha\join (\alpha\to (\beta\join\neg\beta))$.
\item The logic axiomatized relative to $\lgk{LP}_2$ by $(\alpha\to\beta)\join (\beta\to\alpha)\join (\alpha\leftrightarrow\neg\beta)$.
\item The logic axiomatized relative to $\lgk{LP}_2$ by $\neg\alpha\join\neg\neg\alpha$.
\item The trivial logic consisting of all formulas.
\end{enumerate}
\end{theorem}

Maksimova's proof of Theorem~\ref{thm:IPC} depends on the equivalence of \prp{CIP} for any given axiomatic extension $\lgk{L}$ of $\ipc$ with the amalgamation property (or, equivalently for varieties of Heyting algebras, the superamalgamation property) for the associated class of Heyting algebras modeling $\lgk{L}$; see \refchapter{chapter:algebra}. A full account of these results can also be found in the monograph of Gabbay and Maksimova \cite[Chapter~6]{GabMak2005}.

Theorem~\ref{thm:IPC} also has relevance to positive intuitionistic propositional logic $\ipc^+$, i.e., the fragment of $\ipc$ without $\bot$.

\begin{proposition}[\cite{Mak77}]\label{thm:posIPC}
There are exactly four axiomatic extensions of $\ipc^+$ with the \prp{CIP} (or, equivalently, with the \prp{DIP} or \prp{MIP}). These four logics are:
\begin{enumerate}
\item $\ipc^+$ itself.
\item The positive fragment of $\cpl$.
\item The positive fragment of the G\"odel-Dummett Logic $\lgk{LC}$.
\item The trivial logic consisting of all formulas.
\end{enumerate}
\end{proposition}

Further studies of interpolation in $\ipc$ have sharpened it in a number of ways. Of these, work on uniform interpolation in $\ipc$ and its extensions, began by Pitts in \cite{Pitts1992}, has proven especially prominent. This is further discussed in \refchapter{chapter:uniform}. Interpolation and Beth definability for several fragments of intuitionistic logic have recently been studied in \cite{Dekker2020}.

\subsection{Modal Logics}\label{subsec:modal}

In their most common formulation, modal logics expand $\cpl$ by additional unary connectives $\Box$ and $\Diamond$, respectively intended to capture various kinds of `necessity' and `possibility'. Modal logics are hence typically formulated in the language $\Lmod = \{\meet,\join,\to,\bot,\top,\Box,\Diamond\}$ that expands the basic language $\Lbas$ of $\cpl$ and $\ipc$ by symbols for these unary connectives. Considered as an expansion of $\cpl$, modal logics typically validate the equivalence of $\Box\alpha$ and $\neg\Diamond\neg\alpha$, so in fact only one of the connectives $\Box$ or $\Diamond$ need be taken as primitive. However, for modal logics that expand non-classical base logics---for example, intuitionistic modal logics---the connectives $\Box$ and $\Diamond$ need not be interdefinable; see \cite{SimpsonThesis}. Of course, one may also include several modal operators of the $\Box$ type and their dual $\Diamond$ connectives, obtaining multimodal logics. Much can be said about interpolation in the multimodal setting; see \cite[Section~4.5]{GKWZmanydim}. Here, however, we confine our attention to the unimodal case.

In \refchapter{chapter:modal}, we have already seen the basic modal logic $\Kbas$, which can be specified either proof-theoretically or as the logic of all (modal) Kripke frames. Actually, $\Kbas$ is really \emph{two} logics, depending on whether one takes its local consequence relation $\vdash^\ell_\Kbas$ or its global consequence relation $\vdash^g_\Kbas$; see \cite{cha:mod} for relevant definitions. Because deductive interpolation for the local consequence relation is tantamount to Craig interpolation, we will usually just consider the global consequence relation and refer to the \prp{DIP} and the \prp{CIP}, denoting $\vdash^g_\Lk$ simply as $\vdash_\Lk$ for axiomatic extensions of $\Kbas = (\Lmod,\vdash^g_\Kbas)$. Axiomatic extensions of $\Kbas$ are usually called \emph{normal} modal logics.

Every normal modal logic $\Lk$ satisfies a (implicative) local deduction theorem of the form
\[
\Gamma,\alpha\vdash_\Lk \beta \iff (\alpha\meet \Box\alpha\meet\Box^2\alpha\meet\ldots\meet\Box^n\alpha)\to\beta\text{ for some }n\in\mathbb{N},
\]
where $\Box^k\alpha$ abbreviates the expression $\Box\Box\ldots\Box\alpha$, where $\Box$ occurs $k$ times. Thus, for any normal modal logic $\Lk$, the \prp{CIP} implies the \prp{DIP} and the \prp{MIP} for $\vdash_\Lk$. For any normal modal logic $\Lk$, the \prp{DIP} also implies the weak interpolation property \prp{WIP}; see \cite{Mak2005LC}.

\subsubsection{Normal Modal Logics Generally}

Many fundamental normal modal logics arise as the logics of easily described classes of modal Kripke frames. For instance:
\begin{itemize}
\item $\Tmod$ is the extension of $\Kbas$ by the axiom $\Box\alpha\to\alpha$, and is the logic of the class of reflexive frames.
\item $\Dmod$ is the extension of $\Kbas$ by the axiom $\Box\alpha\to\Diamond\alpha$, and is the logic of serial frames.
\item $\KF$ is the extension of $\Kbas$ by the axiom $\Box\alpha\to\Box\Box\alpha$, and is the logic of the class of transitive frames.
\item $\DF$ is the extension of $\Dmod$ by the axiom $\Box\alpha\to\Box\Box\alpha$, and is the logic of transitive and serial frames.
\item $\SF$ is the extension of $\Tmod$ by $\Box\alpha\to\Box\Box\alpha$, and is the logic of the class of frames whose accessibility relation is a preorder.
\item $\Sfiv$ is the extension of $\SF$ by the axiom $\Diamond\alpha\to\Box\Diamond\alpha$, and is the logic of frames whose accessibility relation is an equivalence relation.
\end{itemize}
Each of $\Kbas$, $\Tmod$, $\Dmod$, $\KF$, $\DF$, $\SF$, and $\Sfiv$ has the \prp{CIP}, and hence the \prp{DIP} and the \prp{MIP}. In fact, each of these has the stronger Lyndon interpolation property for implication; see \refchapter{chapter:prooftheory}. Various proofs exist for these elementary results, and the reader may consult the monograph of Gabbay and Maksimova \cite[Chapter 5]{GabMak2005} for detailed information.

One more fundamental normal modal logic demands special attention. Provability logic $\lgk{G}$ is the extension of $\KF$ by the axiom $\Box(\Box\alpha\to\alpha)\to\Box\alpha$, and Smorynski showed in \cite{Smo1978} that $\lgk{G}$ has the \prp{CIP} (hence also the \prp{DIP} and the \prp{WIP}). The logic $\lgk{G}$ has proven quite important to understanding the scope of interpolation in normal modal logics, and Maksimova has established the following result.

\begin{theorem}[{\cite[Theorem 10]{Mak91}}]\label{thm:ext of G}
There are continuum-many axiomatic extensions of $\lgk{G}$ with the \prp{CIP}, and hence the \prp{DIP} and \prp{WIP} as well.
\end{theorem}

The next theorem, regarding weak interpolation, fills out this picture. In the following, $\lgk{S4.1}$ refers to extension of $\SF$ by the McKinsey axiom $\Box\Diamond\alpha\to\Diamond\Box\alpha$.

\begin{theorem}[{\cite[Proposition 1]{Mak2005LC}}]\label{thm:WIP ext}
Any normal modal logic containing $\lgk{S4.1}$ or $\lgk{G}$ has the \prp{WIP}.
\end{theorem}

\subsubsection{Extensions of $\SF$}

While Theorems~\ref{thm:ext of G} and \ref{thm:WIP ext} may give the reader the impression that interpolation is rather common in normal modal logics, it turns out that this is not the case in extensions of $\SF$: Maksimova has shown in \cite{Mak79} that there are just finitely many axiomatic extensions of $\SF$ with the \prp{DIP} or the \prp{CIP}.

\begin{theorem}[\cite{Mak79,GabMak2005}]
Including the trivial logic consisting of all formulas, there are at most 37 axiomatic extensions of $\SF$ with \prp{CIP} and at most 49 axiomatic extensions of $\SF$ with the \prp{DIP}. There are at least 31 axiomatic extensions of $\SF$ with the \prp{CIP} and at least 43 axiomatic extensions of $\SF$ with the \prp{DIP}. In particular, there are at least 12 extensions of $\SF$ that have the \prp{DIP} but not the \prp{CIP}.
\end{theorem}
The proof relies on algebraic techniques linking interpolation to amalgamation (see \refchapter{chapter:algebra}) as well as the G\"odel translation, linking interpolation for extensions of $\SF$ to corresponding extensions of $\ipc$. The question of interpolation remains open for six axiomatic extensions of $\SF$, and the main difficulty seems to be that these logics are non-canonical. This remains one of the longest-standing problems in interpolation.

\begin{question}
How many axiomatic extensions of $\SF$ have the \prp{DIP}? What about the \prp{CIP}?
\end{question}

A full discussion of interpolation in extensions of $\SF$, together with a precise list of the extensions of $\SF$ known to have the \prp{CIP} and the \prp{DIP}, may be found in \cite[Chapter 8]{GabMak2005}.

We note that  Grzegorczyk's logic $\Grz$ is among extensions of $\SF$ with the \prp{CIP}. The logic $\Grz$ is axiomatized relative to $\SF$ by the $\Box(\Box(\alpha\to\Box\alpha)\to\alpha)\to\alpha$, and was first shown to have the \prp{CIP} by Boolos in \cite{Boolos1980}. Interestingly, for extensions of $\Grz$, the \prp{CIP} and the \prp{DIP} coincide; counting the trivial logic, there are seven of these (see \cite[Theorem~8.46]{GabMak2005}).

\subsubsection{Non-Normal Modal Logics}

Non-normal modal logics are those that do not validate the characteristic axiom $\Box (\alpha\to\beta)\to(\Box\alpha\to\Box\beta)$ of the logic $\Kbas$ or otherwise do not validate the rule of necessitation. Dropping these conditions presents a significant challenge, since, for example, the characteristic axiom of $\Kbas$ is validated in all modal Kripke frames, and hence non-normal modal logics cannot be studied using the relational semantical methods typical of normal modal logics.

A paradigmatic example of a non-normal modal logic is given by the basic monotone modal logic, which replaces the aforementioned characteristic axiom of $\Kbas$ by the rule `from $\alpha\to\beta$, derive $\Box\alpha\to\Box\beta$'. Santocanale and Venema have shown in \cite{SanVen2010} that the basic monotone modal logic has a uniform variant of the \prp{CIP}.

Various kinds of interpolation have also been established for several other non-normal modal logics, such as the congruential modal logic $\lgk{E}$ that replaces the characteristic axiom of $\Kbas$ by the rule `from $\alpha\leftrightarrow\beta$, infer $\Box\alpha\leftrightarrow\Box\beta$', which has a uniform Lyndon variant of the \prp{CIP}. See \cite{TabIemJal2021} for this as well as for a recent proof-theoretic study of interpolation in several basic systems of non-normal modal logic.

\subsection{Paraconsistent Logics}\label{subsec:paraconsistent}

The \emph{principle of explosion} stipulates that, for some given logic $\lgk{L}$ with a negation connective $\neg$, we have
\[
\alpha,\neg\alpha\vdash\beta.
\]
`Paraconsistent logic' is an umbrella term most commonly referring to any of several logics that refute the principle of explosion. Such logics are motivated by a number of concerns, most notably the need for reasoning non-trivially in the presence of contradictions and attendant applications. Here we consider just a few of the most prominent families of paraconsistent systems.

\subsubsection{The Logic of Paradox}
Probably the most famous paraconsistent logic is Priest's logic of paradox $\lgk{LP}$; see \cite{Priest1979}. The logic $\lgk{LP}$ is often defined over the language $\Ln_\lgk{LP} = \{\meet,\join,\neg\}$ and is specified algebraically by certain matrix models, as in, e.g., \cite{Milne2016}. Sometimes an implication connective $\to$ is also included in the language, but this connective does not satisfy modus ponens or modus tollens; see \cite{Priest1979}. Considered over the language $\Ln_\lgk{LP}$ and defined according to the usual matrix definitions, $\lgk{LP}$ enjoys the following form of the deductive interpolation property:
\begin{quote}
If $\not\vdash_\lgk{LP} \beta$ and $\alpha\vdash_\lgk{LP}\beta$, then there is a formula $\delta$ such that $\alpha\vdash_\lgk{LP}\delta$, $\delta\vdash_\lgk{LP}\beta$, and $\var{\delta}\subseteq\var{\alpha}\cap\var{\beta}$.
\end{quote}
See \cite{Milne2016} for a discussion placing this interpolation result in the context of so-called split interpolation, and see \cite{Blomet2025} for a more wide-ranging discussion of split interpolation in three-valued logics.

\subsubsection{Johansson's Minimal Logic}\label{subsec:minimal}
Some other paraconsistent logics are rather closer in flavor to $\ipc$. Johansson's minimal logic $\Jhn$ (see \cite{Johansson1937}) is one of these. It can be formulated in the basic language $\Lbas = \{\meet,\join,\to,\bot,\top\}$ according to a Hilbert-style calculus that relaxes some of the conditions of $\ipc$; see \cite[Chapter 2]{Odintsov2008}. Indeed, $\ipc$ itself can be axiomatized relative to $\Jhn$ by adding the single axiom $\bot\to\alpha$. In semantic terms, $\Jhn$ can be characterized as the logic corresponding to \emph{j-algebras}, which are algebraic structures of the form $(A,\meet,\join,\to,\bot,\top)$ such that $(A,\meet,\join,\to,\top)$ is a Brouwerian algebra (aka an implicative lattice), i.e. a $\bot$-free subreduct of a Heyting algebra, and $\bot$ is an arbitrary constant.

Although it is a close cousin of $\ipc$, understanding interpolation in $\Jhn$ has proven very challenging. Among other things, it remains open whether there are finitely many axiomatic extensions of $\Jhn$ with \prp{CIP}, which is equivalent to \prp{DIP} in this setting. There are, however, some results confined to tractable extensions of $\Jhn$, and  we summarize a couple notable landmarks.
\begin{proposition}[{\cite[Corollary 5]{Mak2012}}]
Let $\lgk{JX}$ be the extension of $\Jhn$ by the axiom $(\bot\to\alpha)\join (\alpha\to\bot)$. There are finitely many axiomatic extensions of $\lgk{JX}$ with the \prp{CIP} or, equivalently, the \prp{DIP}.
\end{proposition}
\begin{proposition}[{\cite[Theorem~5.8]{Mak2011}}]
Let $\lgk{Gl}$ be the extension of $\Jhn$ by the law of excluded middle $\alpha\join\neg\alpha$. There are exactly twenty axiomatic extensions of $\lgk{Gl}$ with the \prp{CIP} or, equivalently, the \prp{DIP}.
\end{proposition}
The logic $\lgk{Gl}$ has played a significant role in the study of interpolation in extensions of $\Jhn$. In \cite{Mak2010}, Maksimova shows that the study of the \prp{WIP} for extensions of $\Jhn$ can be reduced to the case of $\lgk{Gl}$.
A complete description of axiomatic extensions of $\Jhn$ with this weak interpolation property is given in \cite{Mak2011b}, which shows that the set of such logics is the union of eight pairwise disjoint intervals in the lattice of axiomatic extensions.

Many questions about interpolation in extensions of $\Jhn$ remain open.

\begin{question}
How many axiomatic extensions of $\Jhn$ have the \prp{CIP} (equivalently, the \prp{DIP)}? Are there infinitely many such extensions?
\end{question}

\begin{question}
Is it possible to give a complete classification of extensions of $\Jhn$ with the \prp{CIP} (equivalently, the \prp{DIP})?
\end{question}
Despite the apparent difficulty of the previous questions, we note that in some sense interpolation in extensions of $\Jhn$ is constrained to certain regions of the lattice of extensions. Odintsov has shown that the lattice of extensions of $\Jhn$ may be partitioned into certain intervals according to a classification based on the intuitionistic and negative companions of these logics; see \cite{Odintsov2001}. Only finitely many of these intervals contain logics with the \prp{CIP}. For the most up-to-date summary of what is currently known, the reader may consult \cite{MakYun2018}.

\subsubsection{Nelson's Constructive Logic With Strong Negation}

Nelson's logic with strong negation \cite{Nelson1964} amounts to the expansion of intuitionistic propositional logic $\ipc$ by an additional negation $\sim$ that satisfies De Morgan's Laws. There are several different variants of this logic, most often formulated in various fragments of the language $\Ln_{\mathsf N} = \{\meet,\join,\to,\sim,\bot,\top\}$. We will focus mostly on Nelson's paraconsistent logic $\NF$, which can be specified as a logic in the language $\Ln_{\mathsf N}$ by a Hilbert-style calculus or using algebraic models; see \cite[Chapter 8]{Odintsov2008}. The fragment of $\NF$ without $\sim$ coincides with $\ipc$. On the other hand, the fragment without $\bot,\top$, denoted $\NFmin$, is also sometimes considered since it dispenses with the intuitionistic negation $\alpha\mapsto\alpha\to\bot$. Nelson's logic also has a much-studied variant that satisfies the principle of explosion, and this logic, denoted $\NEmin$, can be obtained by adding to $\NFmin$ the single axiom $\sim\alpha\to (\alpha\to\beta)$; see \cite[Section~8.1]{Odintsov2001}. The version of $\NEmin$ with both $\sim$ and the intuitionistic negation is denoted $\NE$.

Using algebraic tools, Goranko has obtained a complete classification of axiomatic extensions of $\NE$ with the Craig interpolation property.

\begin{theorem}[{\cite[Theorem 71]{Goranko1985}}]\label{thm:Goranko}
There are exactly fifteen axiomatic extensions of $\NE$ with the \prp{CIP}. In particular, seven of the non-trivial axiomatic extensions of $\NE$ with the \prp{CIP} arise from expanding the seven non-trivial extensions of $\ipc$ with the \prp{CIP} (see Theorem~\ref{thm:IPC}) with a strong negation $\sim$ satisfying $\neg\alpha\to(\alpha\to\beta)$. The remaining seven non-trivial axiomatic extensions of $\NE$ with the \prp{CIP} are the aforementioned logics respectively extended with the axiom $\neg (\alpha\leftrightarrow\sim\alpha)$.
\end{theorem}

Metcalfe has also given an analytic sequent calculus for $\NE$ in \cite{Metcalfe2009}, which can be used to give a different proof that $\NE$ has the \prp{CIP}.

Odintsov has widened Goranko's classification to give a complete classification of axiomatic extensions of $\NF$ with the \prp{CIP}.

\begin{theorem}[{\cite[Theorem~5.9]{Odintsov2006}}]
There are exactly 29 axiomatic extensions of $\NF$ with the \prp{CIP}. Fifteen of these are given by Theorem~\ref{thm:Goranko}. Seven of the remaining logics are given by expanding the seven non-trivial extensions of $\ipc$ with the \prp{CIP} with a strong negation $\sim$. The remaining seven logics are the respective axiomatic extensions of the latter logics by the axiom $\neg\neg (\alpha\join\sim\alpha)$.
\end{theorem}
Note that the axiomatic extensions of $\NF$ have a classical deduction theorem (see \cite{NRMS2019}), so the preceding theorems also given complete classifications for axiomatic extensions with the \prp{DIP}.

\subsection{Fuzzy Logics}\label{subsec:fuzzy}

The moniker \emph{fuzzy logic} typically applies to any of several systems intended to capture reasoning where truth values are vague, fuzzy, or uncertain.\footnote{These are sometimes called `many-valued' logics. However, since most of the non-classical logics considered here arguably have many truth values in various semantic frameworks, we do not use that term here.} Historically, the development of fuzzy logic has been heavily influenced by various efforts to endow the real unit interval $[0,1]$ with logical structure (often algebraic in nature) making it into a set of truth values, where the extrema $0$ and $1$ represent `crisp' or classical truth values. The restriction of the logical structure overlaying $[0,1]$ to the crisp values in $\{0,1\}$ should then recover $\cpl$ according to this program.

Fuzzy logic has undergone several successive phases of generalization over the past century or so, leading to frameworks that progressively encompass the aforementioned ideas in broader and broader environments. We will organize our discussion of interpolation in fuzzy logics more or less in order of increasing generality, touching on the most prominent frameworks. 

\subsubsection{{\L}ukasiewicz's Logic and Its Extensions}

Probably the most thoroughly studied fuzzy logics---and also among the oldest---are the various systems of {\L}ukasiewicz, of which the infinite-valued {\L}ukasiewicz logic $\L$ is the most prominent; see \cite{Luk1920}, and for a historical account see \cite{Cig2007}. The logic $\L$ may be formulated in a number of equivalent languages. For the sake of relating it to other logics that we discuss, we note that it can be formulated in the language $\Lsc = \{\meet,\join,\cdot,\to,0,1\}$, where $\cdot$ gives a kind of `strong conjunction' and the negation $\neg\alpha$, which abbreviates $\alpha\to 0$, satisfies De Morgan's laws as well as the law of double negation $\alpha\leftrightarrow\neg\neg\alpha$. In many sources, instead of $\cdot$ the De Morgan dual connective $\oplus$ is used, i.e., $\alpha\oplus\beta$ abbreviates $\neg (\neg\alpha\cdot\neg\beta)$. The connective $\oplus$ thus represents a kind of `strong disjunction'.

The logic $\L$ has a rather elaborate proof theory (see \cite{MOG08}), so $\L$ and its extensions are usually studied using semantic methods. These most usually involve the equational class of \emph{MV-algebras} (see \cite{Cignoli2000}), which give the equivalent algebraic semantics of $\L$. In particular, $\L$ is sound and complete with respect to a single algebraic model defined on $[0,1]$, where $\meet$ and $\join$ interpret the usual binary maximum and minimum, $\cdot$ is defined by $x\cdot y = \max\{x+y-1,0\}$, where $+$ is the usual sum of real numbers, and $\to$ is uniquely specified by stipulating that $x\to y$ is the greatest element $z$ of $[0,1]$ such that $x\cdot z \leq y$. This single algebraic model generates the entire equational class of MV-algebras.

Every axiomatic extension $\lgk{L}$ of $\L$ has a (implicative) local deduction theorem of the form
\[
\Gamma,\alpha\vdash_\lgk{L}\beta \iff \Gamma\vdash_\lgk{L}\alpha^n\to\beta \text{ for some positive integer }n,
\]
where $\alpha^n$ is an abbreviation for the formula $\alpha\cdot\alpha\cdot\ldots\cdot\alpha$ in which $\alpha$ occurs $n$ times. Thus, for any axiomatic extension $\lgk{L}$ of $\L$, the \prp{CIP} implies the \prp{DIP}. However, as shown by Komori, there are very few axiomatic extensions of $\L$ with the \prp{CIP}.
 
\begin{theorem}[\cite{Komori1981}]
The only axiomatic extensions of $\L$ with the \prp{CIP} are $\cpl$ and the trivial logic consisting of all formulas.
\end{theorem}
On the other hand, there are many more axiomatic extensions of $\L$ with the \prp{DIP}. The following result is due to Di Nola and Lettieri, who proved it in an equivalent algebraic formulation by giving a complete description of equational classes of MV-algebras with the amalgamation property; see \refchapter{chapter:algebra}.

 \begin{theorem}[{cf.~\cite{DiNola2000}}]
The axiomatic extensions of $\L$ that have the \prp{DIP} are exactly those characteristic with respect to a single totally ordered MV-algebra. Thus, in particular, there countably infinitely many of these.
 \end{theorem}
 
Metcalfe, Montagna, and Tsinakis have obtained an analogous, albeit somewhat more complicated, classification for axiomatic extensions of the positive fragment of $\L$ that have the \prp{DIP}; see \cite[Theorem 63]{MMT14}. Like Di Nola and Lettieri, their result is phrased in algebraic terms: They actually give a complete classification of equational classes of Wajsberg hoops that have the amalgamation property, and the result for the \prp{DIP} follows because these algebraic structures provide the equivalent algebraic semantics of positive {\L}ukasiewicz logic. We do not state the classification in full, but note the following consequence.

\begin{proposition}[{cf.~\cite[Theorem 63]{MMT14}}]
Let $\L^+$ be the positive fragment of $\L$. There are countably infinitely many axiomatic extensions of $\L^+$ with the \prp{DIP}.
\end{proposition}

\subsubsection{H\'{a}jek's Basic Fuzzy Logic and Its Extensions}

As discussed above, {\L}ukasiewicz logic is characterized by a single algebraic model constructed on top of the real unit interval $[0,1]$, where the strong conjunction is interpreted by the binary operation $x\cdot y = \max\{x+y-1,0\}$. This operation is associative, commutative, has $1$ as a unit element, and is monotone in each of its coordinates. The binary operations on $[0,1]$ satisfying this latter list of conditions are often called \emph{triangular norms} or \emph{t-norms} for short. Triangular norms have played an extremely important role in the development of approximate reasoning, and the {\L}ukasiewicz t-norm defined by $x\cdot y = \max\{x+y-1,0\}$ has an additional important property: It is continuous as a map $[0,1]\times[0,1]\to [0,1]$, where $[0,1]$ has its usual topology.

Continuous t-norms have a number of nice properties. Among other things, each continuous t-norm has a corresponding residual operation $\to\colon[0,1]\times [0,1]\to [0,1]$ defined by stipulating that $x\to y = \max \{z\in [0,1] \mid x\cdot z\leq y\}$, which is equivalent to stipulating the \emph{residuation law},
\[
x\cdot y\leq z \iff x\leq y\to z.
\]
Further, the operations of binary infimum $\meet$ and binary supremum $\join$ turn out to be definable in the this language too. H\'{a}jek undertook an extensive study of the logics arising in this way from continuous t-norms in \cite{Hajek1998}. There he provided a Hilbert-style axiomatization for his basic fuzzy logic $\BL$ as a logic in the language $\Lsc=\{\meet,\join,\cdot,\to,0,1\}$. Later on, Cignoli, Godo, Esteva, and Torrens showed in \cite{CEGT2000} that H\'{a}jek's axiomatization is complete: $\BL$ is exactly the logic of continuous t-norms. 

Interpolation for $\BL$ and neighboring logics has been the subject of lots of work for several decades now. Although $\BL$ is the logic of all continuous t-norms, arbitrary axiomatic extensions of $\BL$ may not be characterized by continuous t-norms on [0,1]. Among these, Baaz and Veith have shown in \cite{BaazVeith1999} that the only logic with the \prp{CIP} is G\"{o}del-Dummett logic (see Theorem~\ref{thm:IPC}), which may be viewed as the extension of $\BL$ by the axiom $x\leftrightarrow x\cdot x$. In particular, $\BL$ itself does not have the \prp{CIP}. This result has subsequently been strengthened by Montagna \cite{Montagna2006}, who showed that the only axiomatic extensions of $\BL$ with the \prp{CIP} are (definitionally equivalent to) superintuitionistic logics. 

\begin{theorem}[\cite{Montagna2006}, cf.~\cite{BaazVeith1999}]\label{thm:CIP in BL}
The axiomatic extensions of H\'{a}jek's basic fuzzy logic $\BL$ with the \prp{CIP} are exactly:
\begin{enumerate}
\item G\"odel-Dummett logic $\lgk{LC}$, axiomatized relative to $\BL$ by $x\leftrightarrow x\cdot x$.
\item The axiomatic extension of $\lgk{LC}$ characterized by the three-element Heyting algebra.
\item Classical propositional logic $\cpl$.
\item The trivial logic consisting of all formulas.
\end{enumerate}
\end{theorem}

Axiomatic extensions of $\BL$ have the same local deduction theorem as $\L$, i.e., every axiomatic extension $\lgk{L}$ of $\BL$ satisfies
\[
\Gamma,\alpha\vdash_\lgk{L}\beta \iff \Gamma\vdash_\lgk{L}\alpha^n\to\beta \text{ for some positive integer }n,
\]
so the \prp{CIP} implies the \prp{DIP} in this setting, and the \prp{DIP} is equivalent to the \prp{MIP}. However, the \prp{DIP} need not imply the \prp{CIP} for an axiomatic extension of $\BL$. Indeed, the situation is much more complicated for deductive interpolation. The first progress on this was made by Montagna in 2006.

\begin{theorem}[\cite{Montagna2006}]
H\'{a}jek's basic fuzzy logic $\BL$ has the \prp{DIP}, but there are uncountably many axiomatic extensions of $\BL$ that do not have the \prp{DIP}.
\end{theorem}
Montagna did not obtain a complete description of axiomatic extensions of $\BL$ with the \prp{DIP}, or even determine whether there are countably many of these (there are at infinitely many of them since $\L$ is an axiomatic extension of $\BL$). Subsequent work by Cortonesi, Marchioni, and Montagna \cite{CMM11} (applying techniques from first-order model theory) and Aguzzoli and Bianchi \cite{AB21,AB23} (applying more algebraic techniques) significantly deepened our understanding of the \prp{DIP} in axiomatic extensions of $\BL$, and a complete description of these was finally obtained by Fussner and Santschi in \cite{FS2024BL}. This description turns out to be rather elaborate, but we summarize a few of its salient features in the following theorem.

\begin{theorem}[{\cite{FS2024BL}}]
There are only countably many axiomatic extensions of $\BL$ with the \prp{DIP}. Indeed, the subposet of these in the lattice of all axiomatic extensions of $\BL$ is the disjoint union of countably infinitely many finite intervals. The same is true of the positive fragment of $\BL$.
\end{theorem}
Indeed, the axiomatic extensions of $\BL$ with the \prp{DIP} are explicitly described in \cite{FS2024BL} in terms of their characteristic generating models. Although this description is quite explicit, an explicit axiomatization of these countably infinitely many logics is not known.
\begin{question}
Is it possible to give explicit axiomatizations for the extensions of $\BL$ with the \prp{DIP}? Are all of them finitely axiomatizable?
\end{question}
The description of extensions of $\BL$ with the \prp{DIP} from \cite{FS2024BL} draws on techniques from regular languages, and suggests that there is some decidability result lurking in the background. In connection with the previous question, one might wonder about the following point.
\begin{question}
Is the problem of determining whether a given finitely axiomatized extension of $\BL$ has the \prp{DIP} effectively decidable?
\end{question}
The previous question amounts, of course, to whether it is possible to determine algorithmically whether a given finitely axiomatized extension of $\BL$ is one of the logics listed in \cite{FS2024BL}.

\subsubsection{Monoidal T-norm Based Logics}

One of the most notable features of continuous t-norms, considered as opposed to arbitrary t-norms, is that each continuous t-norm $\cdot$ has a residual defined by $x\to y = \max \{z\in [0,1] \mid x\cdot z\leq y\}$. This fact is not a consequence of full continuity, but instead only involves the lower semicontinuity of the t-norm $\cdot$. Consequently, logics associated to lower semicontinuous t-norms have recently attracted considerable attention. The first proper study of this topic was \cite{EstGo2001}, which introduced the monoidal t-norm based logic $\mtl$ as a logic defined over the basic language $\{\meet,\join,\cdot,\to,0,1\}$. This logic is characteristic with respect to arbitrary algebras of lower semicontinuous t-norms, and $\BL$ itself is the axiomatic extension of $\mtl$ by the divisibility condition $x\cdot (x\to y) \leftrightarrow x\meet y$.

To date, there is not much systematic information about interpolation in axiomatic extensions of $\mtl$ that go beyond the case of $\BL$ extensions. Analogous proofs as those given by Baaz and Veith \cite{BaazVeith1999} and Montagna \cite{Montagna2006} show that the only axiomatic extensions of $\mtl$ with the \prp{CIP} are the axiomatic extensions of $\BL$ with the \prp{CIP}, i.e., those summarized in Theorem~\ref{thm:CIP in BL}. By showing the failure of amalgamation in certain associated varieties, it can be shown with small algebraic models that $\mtl$ does not have the \prp{DIP} either; see \cite{FusSan2024semlin,UgolGiust2024}, but we have not yet identified any minimal axiomatic extensions of $\mtl$ with the \prp{DIP}. Many other open questions remain in this vicinity, e.g., the cardinality of the set of axiomatic extensions of $\mtl$ with the \prp{DIP} is not known.

\begin{question}
Is there a minimum axiomatic extension of $\mtl$ with the \prp{DIP}? Equivalently, is there a largest variety of MTL-algebras with the amalgamation property?
\end{question}

\begin{question}
Are there uncountably many axiomatic extensions of $\mtl$ with the the \prp{DIP}?
\end{question}

\subsection{Relevant Logics}\label{subsec:relevant}

Relevant logics arise from the intuition that, in any logically valid implication $\alpha\to\beta$, the antecedent $\alpha$ and the consequent $\beta$ ought to be germane to one another. This is often made precise by requiring the \emph{variable sharing property}, which demands that for a logic $\lgk{L}$,
\[
\vdash_\lgk{L} \alpha\to\beta \Rightarrow \var{\alpha}\cap\var{\beta}\neq\emptyset.
\]
The main line of research on relevant logics descends from the school of Anderson and Belnap, represented, for example, by the monographs \cite{AB1,AB2}. Probably the most prominent logical system in this family is the logic $\R$, which is often specified either in the language $\{\meet,\join,\to,\neg\}$ or $\{\meet,\join,\cdot,\to,\neg\}$. The operation $\cdot$ represents a form of conjunction often called \emph{cotenability}, and may be defined from the others by stipulating that $\alpha\cdot\beta$ abbreviates $\neg(\neg\alpha\to\beta)$. The most extensively studied axiomatic extension of $\R$ is probably the logic R-mingle $\RM$, axiomatized relative to $\R$ itself by adding
\[
\alpha\to (\alpha\to\alpha),
\]
an axiomatic form of the mingle rule (see Section~\ref{sec:substructural logics}).

Various technical concerns have also motivated the addition of logical constants---often called \emph{Ackermann constants} in the relevantist literature---$\mathsf{t}$ and $\mathsf{f}$ to the language. When this is done, $\mathsf{t}$ is usually stipulated to be a neutral element for cotenability and $\mathsf{f}$ is its negation. The logics in this expanded signature are often denoted by $\Rt$ and $\RMt$, respectively. All of these logics can be variously specified by Hilbert-style systems or algebraic semantics in terms of so-called \emph{De Morgan monoids} \cite{AB1}, or by relational semantics that employ a ternary accessibility relation \cite{RM1972}. In the case of R-mingle and its variants, the accessibility relation can be taken to be binary; see \cite{Dunn1976,FG2019}.

The presence of the Ackermann constants makes a big difference for these logics: $\R$ does not have a local deduction theorem, but $\Rt$ has an explicit deduction theorem of the form
\[
\Gamma,\alpha\vdash_{\Rt}\beta \iff \Gamma\vdash (\alpha\meet\mathsf{t})\to\beta.
\]
The same deduction theorem holds for axiomatic extensions of $\Rt$, including $\RMt$. On the other hand, $\RM$ has an explicit deduction theorem with the non-standard form
\[
\Gamma,\alpha\vdash_\RM \beta \iff (\neg(\alpha\to\neg\beta)\join (\alpha\to\beta))\meet (\neg\alpha\join\beta),
\]
a result due to Tokarz; see \cite[p.~126]{CzeProto}.

Urquhart showed in \cite{Urq1993} that the \prp{CIP} fails for a huge number of relevant logics, notably including $\R$ and $\Rt$, as well as the related systems $\lgk{E}$ of entailment and $\lgk{T}$ of ticket entailment. Urquhart's proof used techniques from projective geometry to construct some ternary-relational models witnessing the failure of amalgamation (see \refchapter{chapter:algebra}) from which an implication serving as a counterexample for the \prp{CIP} could be extracted. As an artifact of these geometric techniques, the counterexample witnessing failure of the \prp{CIP} was very large: The largest of the three algebras giving the amalgamation failure had cardinality exceeding $10^{247}$, and the implication showing the failure of \prp{CIP} had $14$ variables. Later on, by exhibiting the failure of Beth definability, Urquhart gave a much simpler proof of the failure of $\prp{CIP}$ in $\Rt$ and some other logics in the vicinity, showing that the valid three-variable implication
\[
\vdash_{\Rt}[\mathsf{t}\meet (q\to\mathsf{f})\meet ((q\meet p)\to\mathsf{t})]\to \{[\mathsf{t}\meet (\mathsf{t}\to r)\meet (\mathsf{f}\to (p\join r))]\to (q\to r)\},
\]
has no Craig interpolant; see \cite{Urq99,RRR2}.

Prior to Urquhart's work, Anderson and Belnap showed at \cite[pp. 416-417]{AB1} that $\RM$ lacks \prp{CIP} in the following formulation:
\begin{equation*}
  \parbox{\dimexpr\linewidth-6em}{%
    \strut
     Whenever $\not\vdash_\RM\neg\alpha$, $\not\vdash_\RM\beta$, and $\vdash_\RM \alpha\to\beta$, then there exists a formula $\delta$ such that $\vdash_\RM \alpha\to\delta$, $\vdash_\RM \delta\to\beta$, and $\var{\delta}\subseteq\var{\alpha}\cap\var{\beta}$.
    \strut
  }
\end{equation*}
Avron has introduced another implication-like connective $\supset$ in $\RM$, defined in terms of the primitive connectives by
\[
\alpha\supset\beta = (\alpha\to\beta)\join\beta
\]
Avron has shown in \cite{Avr1986} that this connective gives a classical deduction theorem for $\RM$ and that it satisfies the following interpolation property:
\begin{quote}
If $\vdash_{\RM}\alpha\supset\beta$ and $\not\vdash_{\RM} \beta$, then there is a formula $\delta$ such that $\vdash_\RM\alpha\to\delta$, $\vdash_\RM\delta\to\beta$, and $\var{\delta}\subseteq\var{\alpha}\cap\var{\beta}$.
\end{quote}
Returning to the case with logical constants $\mathsf{t}$ and $\mathsf{f}$, Meyer reportedly showed in \cite{Mey80} that $\RMt$ has the \prp{CIP}. Much later, Marchioni and Metcalfe gave an exhaustive classification of the axiomatic extensions of $\RMt$ with the \prp{CIP}, using a combination of techniques from algebra and model theory, in particular tools from quantifier elimination.
\begin{theorem}[{\cite{MarMet2012}}]
Including the trivial logic consisting of all formulas, there are exactly nine axiomatic extensions of $\RMt$ with the \prp{CIP}. Further, the \prp{CIP} and the \prp{DIP} coincide for axiomatic extensions of $\RMt$, so these are also exactly the axiomatic extension with the \prp{DIP}.
\end{theorem}
The fact that the \prp{CIP} and the \prp{DIP} coincide in this setting is especially notable since, unlike the many of the previously considered cases that exhibit this behavior, $\RMt$ does not have a classical deduction theorem for the usual implication.

The logic $\RMt$ is characterized by linearly ordered De Morgan monoids called \emph{Sugihara monoids}, but Sugihara monoids do not exhaust the linearly ordered algebraic models of $\Rt$. The class of all semilinear De Morgan monoids---that is, those De Morgan monoids that are subalgebras of direct products of linearly ordered ones---characterizes an axiomatic extension $\SemRt$ of which the class of all semilinear De Morgan monoids is the equivalent algebraic semantics. The following result contrasts against the apparent scarcity of relevant logics with interpolation.

\begin{proposition}[{\cite[Proposition 5.3]{FusSan2024semlin}}]
The logic $\SemRt$ does not have the \prp{DIP} and hence does not have the \prp{CIP} either, but there are infinitely many axiomatic extensions of $\SemRt$ that have the \prp{DIP}.
\end{proposition}
We do not yet have a complete classification of axiomatic extensions of $\Rt$ with the \prp{DIP} or the \prp{CIP}, nor even such a classification for $\SemRt$. Obtaining such a complete classification likely demands a more refined structural understanding of De Morgan monoids.

\begin{question}
Is it possible to classify the extensions of $\SemRt$ with the \prp{DIP}? What about the \prp{CIP}? What about for $\Rt$, $\R$ and $\RM$?
\end{question}

Urquhart's many results are often wrongly summarized as saying that relevant logics lack interpolation. As we have seen above, the situation is more complicated. To the extent that the variable sharing property can be taken as a hallmark for relevance, the following question is raised.

\begin{question}
What is the precise relationship between interpolation and the variable sharing property?
\end{question}

\subsection{Substructural Logics}\label{sec:substructural logics}

Many of the logics discussed in the previous sections can be regarded as \emph{substructural logics}. There is no mathematically precise demarcation of this family of logics, but they are unified by the idea of dropping or restricting the application of the structural rules of exchange $\pfa{\exr}$, left weakening $\pfa{\wlr}$, right weakening $\pfa{\wrr}$, contraction $\pfa{\cnr}$, and mingle $\pfa{\mglr}$ that are validated in the Gentzen calculus for $\ipc$; see Figure~\ref{fig:structural} and also \refchapter{chapter:prooftheory}. Most often, this involves including a `strong conjunction' connective $\cdot$ in the basic language, which simulates comma as it appears in Gentzen calculi. Further modification of the basic signature of $\ipc$ is necessitated in some cases. For example, comma is non-commutative when the exchange rule $\pfa{\exr}$ is dropped, and this presents the possibility of having separate modus ponens rules from the left and the right. Thus, we may have two separate implication connectives $\under$ and $\ovr$ linked to these two modus ponens rules: $\alpha, \alpha\under \beta \seq \beta$ and $\beta\ovr \alpha, \alpha \seq \beta$. One may also include a truth constant $\mathsf{e}$ and a falsity constant $\mathsf{f}$. For our purposes here, we take the basic language of substructural logics to be $\Ln_\FL = \{\meet,\join,\cdot,\under,\ovr,\mathsf{e},\mathsf{f}\}$. We will also consider the positive (i.e., $\mathsf{f}$-free) language $\Ln_\FL^+$.

\begin{figure}[t]
\begin{center}
\fbox{
{
\begin{minipage}{13cm}
\begin{align*}
\begin{array}{cccc}
\infer[\pfa{\exr}]{\Ga_1,\Pi_2,\Pi_1,\Ga_2\seq\De}{\Ga_1,\Pi_1,\Pi_2,\Ga_2\seq\De} & &
\infer[\pfa{\cnr}]{\Ga_1,\Pi,\Ga_2\seq\De}{\Ga_1,\Pi,\Pi,\Ga_2\seq\De}\\[.15in]
\infer[\pfa{\wlr}]{\Ga_1,\Pi,\Ga_2\seq\De}{\Ga_1,\Ga_2\seq\De} & &
\infer[\pfa{\wrr}]{\Ga_1,\Pi,\Ga_2\seq\De}{\Pi\seq}
\end{array}\\[.15in]
\infer[\pfa{\mglr}]{\Ga_1,\Pi_1,\Pi_2,\Ga_2\seq\De}{\Ga_1,\Pi_1,\Ga_2\seq\De &\Ga_1,\Pi_2,\Ga_2\seq\De} \qquad
\end{align*}
\caption{The basic structural rules of proof theory, as formulated in \cite{MPT23}. Here each $\Ga_i$, $\Pi$, $\De$, and so on denote finite sequences of formulas, with at most one formulas appearing to the right of $\seq$.}
\label{fig:structural}
\end{minipage}
}}
\end{center}
\end{figure}

Many well-studied substructural logics can be understood as axiomatic extensions of the full Lambek calculus $\FL$, a logic arising from an analytic Gentzen-style sequent calculus formulated in the language $\Ln_\FL$ that omits all of the basic structural rules appearing in Figure~\ref{fig:structural}. Many of the logics we have considered---for instance, $\cpl$, $\ipc$, $\L$, $\BL$, $\mtl$, $\Rt$, $\RMt$, and all axiomatic extensions of these---can be realized in an equivalent formulation as axiomatic extensions of $\FL$. In the latter axiomatic extensions---and indeed in any axiomatic extensions of $\FL$ validating the exchange rule---one may prove that $\alpha\under\beta$ and $\beta\ovr\alpha$ are interderivable. Thus, in the presence of exchange, one may get away with only a single implication in the language, taking us back to the more familiar formulation of several logics with exchange that we have already seen.

The logic $\FL$ and its axiomatic extensions also admit an equivalent algebraic semantics in terms of generalizations of Heyting algebras called \emph{FL-algebras}. Similar remarks apply to the positive fragment of $\FL$, which is algebraized by the equational class of \emph{residuated lattices}. See the monographs \cite{MPT23,GJKO07} for details, along with \refchapter{chapter:algebra}.

Since $\FL$ has two distinct implications, it isn't immediately clear what we mean by the \prp{CIP} in this context. Thankfully, one may show that having \prp{CIP} with respect to $\under$ is equivalent to having the \prp{CIP} with respect to $\ovr$, for any axiomatic extension of $\FL$. Hence we may speak of the \prp{CIP} for some given axiomatic extension of $\FL$ tout court.

Crucially, $\FL$ does not even have a local deduction theorem. It admits a parameterized local deduction theorem (see e.g. \cite[Section 2.4]{GJKO07}), but this kind of deduction theorem is insufficient for the usual proof that the \prp{CIP} implies the \prp{DIP} that holds in many logics. It is shown at \cite[Proposition 10]{KO10} that the \prp{CIP} and the \prp{MIP} are independent for axiomatic extensions of $\FL$, but it is not currently known whether the \prp{CIP} implies the \prp{DIP} for arbitrary axiomatic extensions of $\FL$.

\begin{question}
Let $\lgk{L}$ be an axiomatic extension of $\FL$. If $\lgk{L}$ has the \prp{CIP}, does $\lgk{L}$ have the \prp{DIP}?
\end{question}

The lack of a useful deduction theorem for $\FL$ presents many challenges to the study of interpolation for $\FL$ and its extensions. In the presence of the exchange rule $\pfa{\exr}$, the deduction theorem for $\FL$ simplifies considerably and axiomatic extensions with exchange are much better behaved from the perspective of interpolation. This is discussed further in Subsection~\ref{subsec:exchange} below.

\subsubsection{Interpolation and the Basic Structural Rules}

Aside from $\FL$ itself, the most fundamental axiomatic extensions of $\FL$ arise from adding basic structural rules, as in Figure~\ref{fig:structural}, to the Gentzen calculus defining $\FL$.\footnote{Note that these extensions really are axiomatic: The exchange rule is captured by the axiom $(\alpha\cdot\beta)\to (\beta\cdot\alpha)$, the contraction rule by $\alpha\to (\alpha\cdot\alpha)$, the left weakening rule by $\alpha\to\mathsf{e}$, the right weakening rule by $\mathsf{f}\to\alpha$, and the mingle rule by $(\alpha\cdot\alpha)\to\alpha$.} For the basic structural rules appearing in Figure~\ref{fig:structural}, these are denoted by adding subscripts naming the added structural rules so that, for example, $\FLe$ denotes $\FL$ with the exchange rule, $\FLec$ denotes $\FL$ with both the exchange and contraction rules, and so forth. The two weakening rules $\pfa{\wlr}$ and $\pfa{\wrr}$ are often considered together, so that $\FLw$ denotes the addition of both of them. The logic $\FLecw$ is a definitional variant of $\ipc$.

Because the aforementioned logics have well-behaved analytic sequent calculi, the \prp{CIP} can be established for many of them using proof-theoretic techniques grounded in Maehara's method; see \refchapter{chapter:prooftheory}. The following results from the application of this proof strategy.

\begin{theorem}[{See \cite{OK1985}, \cite[Section 5.1.3]{GJKO07}}]\label{thm:basic rules}
Each of $\FL$, $\FLe$, $\FLw$, $\FLew$, and $\FLec$ has the \prp{CIP}.
\end{theorem}

Due to the lack of a local deduction theorem, we do not know whether $\FL$ and $\FLw$ have the \prp{DIP}. Santschi and Jipsen have recently shown in \cite{SJ_RL} that the equational class of FL-algebras does not have the amalgamation property and hence that $\FL$ does not have the Robinson property (see \refchapter{chapter:algebra}), but in the context of $\FL$ this appears to be insufficient to show that the \prp{DIP} fails.

\begin{question}
Does $\FL$ have the the \prp{DIP}? 
\end{question}

For logics with exchange, however, the situation is simpler.

\subsubsection{Logics with Exchange}\label{subsec:exchange}

Although general axiomatic extensions of $\FL$ have only a parameterized local deduction theorem, every axiomatic extension $\lgk{L}$ of $\FLe$ has an implicative local deduction theorem of the following form:
\[
\Gamma,\alpha\vdash_\lgk{L}\beta \iff \Gamma\vdash_\lgk{L}(\alpha\meet\mathsf{e})^n\to\beta \text{ for some positive integer }n.
\]
Consequently, for each axiomatic extension $\lgk{L}$ of $\FLe$, the \prp{CIP} implies the \prp{DIP} and so, by Theorem~\ref{thm:basic rules}, each of $\FLe$, $\FLew$, and $\FLec$ has the \prp{DIP}. Moreover, because extensions of $\FLe$ are conjunctive, the \prp{DIP} is equivalent to the \prp{MIP} in this setting. The local deduction theorem also eases the application of algebraic techniques in the study of the \prp{DIP} for extensions of $\FLe$---in particular allowing proof strategies based on amalgamation as in \refchapter{chapter:algebra}---and there has commensurately been quite a lot of success in studying the \prp{DIP} in various axiomatic extensions of $\FLe$. We have already seen this in the case of $\L$, $\BL$, $\Rt$, and $\RMt$, all of which can be realized as extensions of $\FLe$.

Girard's linear logic \cite{girard87}, in both its classical and intuitionistic variants, can also be understood as arising from the calculus for $\FLe$ by including additional constants $\bot$ and $\top$ as well as `exponentials' $!$ and $?$, which represent certain kinds of $\SF$-like modal box connectives. Roorda has shown that the \prp{CIP} for these systems, as well as several fragments thereof, can also be established using the proof-theoretic methods previously mentioned; see \cite{Roorda1994}. Intuitionistic and classical linear logic also have implicative local deduction theorems due to Avron \cite{Avr88}, so Roorda's result entails the \prp{DIP} for both classical and intuitionistic linear logic as well. Indeed, using algebraic tools, Fussner and Santschi exhibited uncountably many axiomatic extensions of each of $\FL$, classical linear logic, and intuitionistic linear logic with the \prp{DIP}.

\begin{theorem}[{\cite{FSLinear}}]
There are continuum-many axiomatic extensions of $\FLe$ with the \prp{DIP}. The same holds of the expansion of $\FLe$ by constants designating bounds $\bot$ and $\top$, as well as by the exponentials $!$ and $?$. In particular, each of intuitionistic linear logic and classical linear logic has continuum-many axiomatic extensions with the \prp{DIP}.
\end{theorem}
Notably, the axiomatic extensions constructed in \cite{FSLinear} do not have the \prp{CIP}. The number of axiomatic extensions of $\FLe$ with the \prp{CIP} is apparently unknown. However, since there are continuum-many axiomatic extensions of $\FLe$ without the \prp{DIP}, there are also continuum-many axiomatic extensions of $\FLe$ without the \prp{CIP}.

\begin{question}
How many axiomatic extensions of $\FLe$ have the \prp{CIP}? Are there infinitely many? Uncountably many?
\end{question}

To close our discussion of logics with exchange, we note that Abelian logic---typically understood as the logic of Abelian lattice-ordered groups but also admitting a nice proof theory, as in \cite{MOG08}---can be also be realized as an axiomatic extension of $\FLe$. Because Abelian lattice ordered groups have the amalgamation property, it can also be shown that Abelian logic has the \prp{DIP}. See \cite{FSLinear} for a discussion of amalgamation in Abelian lattice-ordered groups in the context of residuated lattices.

\subsubsection{Logics Without Exchange}\label{subsec:notexchange}

As the previous sections attest, the \prp{DIP} is reasonably common among extensions of $\FLe$. Until recently, it was thought that this was not the case among logics without exchange. Indeed, quite recently it was thought that the exchange rule $\pfa{\exr}$ may be derivable in every axiomatic extension of $\FL$ with the \prp{DIP}, cf. \cite[Problem 5]{GLT15}. This was shown to not be the case by Gil-F\'{e}rez, Jipsen, and Metcalfe in \cite{GilJipMet2020}, which exhibited an equational class of idempotent residuated lattices with the amalgamation property that has non-commutative members. This equational class gives the equivalent algebraic semantics of a certain extension of $\FLcm$ in which $\pfa{\exr}$ is not derivable and which has the \prp{DIP}. The aforementioned equational class is \emph{semilinear} in the sense that its members consist of FL-algebras that are subalgebras of direct products of totally ordered FL-algebras, i.e., the equational class is generated by totally ordered algebras.

As in the case of De Morgan monoids (see Section~\ref{subsec:relevant}), the equational class of FL-algebras generated by the linearly ordered FL-algebras characterizes an axiomatic extension $\SemFL$ of $\FL$. It was shown in \cite{FG1} that axiomatic extensions of $\SemFLcm$ have a local deduction theorem of a non-standard form. Using this fact, \cite{FG2} gave several more natural examples of axiomatic extensions of $\FL$ that have the \prp{DIP} and in which $\pfa{\exr}$ is not derivable. In fact, these results, as well as those of \cite{GilJipMet2020}, actually apply not only to $\SemFLcm$, but also its $\mathsf{f}$-free fragment $\SemFL^+_\mathsf{cm}$.

Recently, by building on \cite{FG1,FG2}, Fussner, Metcalfe, and Santschi proved the following result. 

\begin{theorem}[{\cite[Theorem A]{FMS2024}}]
There are continuum-many axiomatic extensions of $\SemRLcm$ with the \prp{DIP} in which exchange is not derivable.
\end{theorem}
Note that exchange is derivable in $\FLcw$, which constitutes a definitionally equivalent formulation of $\FLecw$ (i.e., of $\ipc$). On the other hand, the mingle rule $\pfa{\mglr}$ may be viewed as a slightly less powerful variant of weakening, and exchange is not derivable in $\FLcm$ or even in $\SemFLcm$. The previous theorem is thus rather striking in comparison to the case of $\ipc$, since $\FLcm$ is one natural candidate for a logic resulting from $\ipc$ by `dropping' exchange. On the other hand, adding exchange to $\SemFLcm$ results in the number of axiomatic extensions with the \prp{DIP} dropping from uncountably many to only finitely many.
\begin{theorem}[{\cite[Theorem B]{FMS2024}}]
There are exactly 60 axiomatic extensions of $\SemRLecm$ with the \prp{DIP}.
\end{theorem}
Thus, in at least some contexts, the exchange rule is actually an impediment to deductive interpolation. This stands in contrast to previous expectations that there may be no logics without exchange that have the \prp{DIP}. Note also that $\SemRLecm$ properly contains the positive fragment of G\"odel-Dummett logic, so the previous theorem can also be contrasted against the results of Maksimova discussed in Section~\ref{subsec:intuitionistic}.

The previous pair of theorems concern the positive logics $\SemRLcm$ and $\SemRLecm$. If we return the falsity constant $\mathsf{f}$ to the signature, we of course still retain continuum-many axiomatic extensions of $\SemFLcm$ with the \prp{DIP} and in which exchange cannot be derived. On the other hand, $\SemFLecm$ has many more extensions with \prp{DIP} than $\SemFLcm$. The exact number of these is not known, but it is shown in \cite[Proposition~5.14]{FMS2024} that $\SemFLecm$ has finitely many axiomatic extensions with the \prp{DIP} and that there are at least $60^4$ of these.

Many questions about interpolation in extensions of $\FL$ remain open. It is still unknown whether there are uncountably many axiomatic extensions of $\FLew$ with the \prp{DIP}, and relatively little is known about extensions of $\FL$ with the \prp{CIP}, aside from the few previously mentioned cases that are settled by application of Maehara's method.

\begin{question}
Are there uncountably many axiomatic extensions of $\FLew$ with the \prp{DIP}?
\end{question}

\begin{question}
Are there infinitely many axiomatic extensions of $\FLew$ with the \prp{CIP}?
\end{question}

\section*{Acknowledgments}
\addcontentsline{toc}{section}{Acknowledgments}

The work contained in this chapter was supported by the Czech Science Foundation project 25-18306M, INTERACT. The author would like to thank Johan van Benthem, Nick Bezhanishvili, and Han Gao for their many useful comments on a draft of this chapter.

%

\providecommand{\noopsort}[1]{}

\end{document}